\documentclass[a4paper,11pt]{article}
\usepackage{microtype}
\usepackage{graphicx}
\usepackage{latexsym}
\usepackage{amsmath, amsthm, amssymb, amsfonts, mathtools}
\usepackage{MnSymbol}
\usepackage{rotating}
\newtheorem{theorem}{Theorem}
\newtheorem{corollary}{Corollary}
\newtheorem{lemma}{Lemma}

\theoremstyle{definition}
\usepackage[hidelinks]{hyperref}
 \hypersetup{
  colorlinks  = true, %Colours links instead of ugly boxes
  urlcolor    = blue, %Colour for external hyperlinks
  linkcolor   = blue, %Colour of internal links
  citecolor   = red   %Colour of citations
}
\usepackage{float}
\usepackage{subfloat}
\usepackage[config, labelfont={sf,bf},textfont=sf]{caption,subfig}  %Have multiple images side by side, each individually captioned.
\usepackage{keyval2e}  %subfig requires this package
\usepackage{everysel,ragged2e} %optional packages for subfig for justification options
\newcommand{\footremember}[2]{%
\footnote{#2}
\newcounter{#1}
\setcounter{#1}{\value{footnote}}%
}

\title{The Propus Construction for Symmetric Hadamard Matrices}

\author{Jennifer Seberry \footremember{UoW}{Centre for Computer Security Research, School of Computer Science and Software Engineering, EIS, University of Wollongong, NSW 2522, Australia. Email: \url{jennifer_seberry@uow.edu.au}} and N. A. Balonin \footremember{Nick}{Saint Petersburg State University of Aerospace Instrumentation,
67, B. Morskaia St., 190000, St. Petersburg, Russian Federation. Email: \url{korbendfs@mail.ru}}}

\begin{document}
\date{29 November 2015}
%\begin{frontmatter}

\maketitle

\begin{abstract}
\textit{Propus} (which means twins) is a construction method for orthogonal $\pm 1$ matrices based on a variation of the Williamson array called the \textit{propus array}

\[ \begin{matrix*}[r]
   A&  B  &  B  & D \\
   B&  D  & -A  &-B \\
   B& -A  & -D  & B \\
   D& -B  &  B  &-A. 
   \end{matrix*} \]

This construction designed to find symmetric Hadamard matrices was originally based on circulant symmetric $\pm 1$ matrices, called \textit{propus matrices}. We also give another construction based on symmetric Williamson-type matrices.

We give  constructions  to find symmetric propus-Hadamard matrices for 57  orders $4n$, $n < 200$ odd. 

We give variations of the above array to allow for more general matrices than symmetric Williamson propus matrices. One such is  the \textit{ Generalized Propus Array (GP)}. 

\end{abstract}


\section{Introduction}
\label{sec:Introduction} Hadamard matrices arise in statistics, signal processing, masking, compression, combinatorics, weaving, spectroscopy and other areas. They been studied extensively. Hadamard showed \cite{JH1893} the order of an Hadamard matrix must be 1, 2 or a multiple of $4$. Many constructions for $\pm 1$ matrices and similar matrices such as Hadamard matrices, weighing matrices, conference matrices and $D$-optimal designs use skew and symmetric Hadamard matrices in their construction. For more details see Seberry and Yamada \cite{SY92}.

An Hadamard matrix of order $n$ is an $n \times n$ matrix with elements $\pm 1$ such
that $H H^\top = H^\top H = n I_n$, where $I_n$ is the $n \times n$
identity matrix and $\top$ stands for transposition. A skew Hadamard matrix $H=I+S$ has $S^{\top } = -S$. For more details see the books and surveys of Jennifer Seberry (Wallis) and others \cite{SY92, WWASJW1972} cited in the bibliography.

Theorems of the type \textit{for every odd integer $n$ there exists a $t$ dependent on $n$ so that Hadamard, regular Hadamard, co-cyclic Hadamard and some full orthogonal designs exist for all orders $2^t n$, $t$ integer} are known \cite{JSW76,craigen95,chaderpourkharaghani14,deLKhar09}.  A similar result for symmetric Hadamard and skew-Hadamard matrices has not yet been published but is conjectured.

Propus is a construction method for orthogonal $\pm 1$ matrices, $A$, $B=C$, and $D$, where \[AA^{\top} + 2BB^{\top} +  DD^{\top} = \text{~constant~} I,\] 
 $I$ the identity matrix,  based on the array
\[ \begin{matrix*}[r]
   A&  B  &  B  & D \\
   B&  D  & -A  &-B \\
   B& -A  & -D  & B \\
   D& -B  &  B  &-A. 
   \end{matrix*} \]

This construction, based on circulant symmetric $\pm 1$ matrices, called \textit{propus matrices}, gives symmetric Hadamard matrices. It also gives aesthetically pleasing visual images (pictures) when converted using MATLAB (we show some below).

We give methods to find propus-Hadamard matrices: using Williamson matrices and $D$-optimal designs. These are then generalized to allow non-circulant and/or non-symmetric matrices with the same aim 
to give symmetric Hadamard matrices.

We show that for
\begin{itemize}
\item $q \equiv  1 \pmod{4}$, a prime power, such matrices exist for order $t = \frac12(q+1)$, and
      thus propus-Hadamard matrices of order $2(q+1)$;
\item $t \equiv3 \pmod{4}$, a prime, such that $D$-optimal designs, constructed using two circulant
      matrices, one of which must be circulant and symmetric, exist of order $2t$, then such propus-Hadamard matrices exist for order $4t$.
\item $4-\{t;s_1,s_2,s_3,s_4;\frac{\sum_{i=1}^4 s_i(s_i -1)}{t-1}\}$ sds, $4t=a^2+b^2+c^2+d^2, a \equiv b = c \equiv d \equiv t \pmod{4},~a = 2s_1 -t, b= 2s_2 -t,~c=2s_3 -t,~d=2s_4 -t$, where one of the 
      supplementary difference sets is symmetric then such propus-Hadamard matrices exist for order $4t$.
\item symmetric variant propus matrices \cite{SY92} may be used to find symmetric propus-Hadamard
      matrices for orders described below (see Corollary \ref{cor:impt}). 
      \end{itemize} 
     
We note that appropriate \textit{Williamson type} matrices may also be used to give propus-Hadamard matrices but do not pursue this avenue in this paper. There is also the possibility that this propus construction may lead to some insight into the existence or non-existence of symmetric conference matrices for some orders.
We refer the interested reader to \url{mathscinet.ru/catalogue/propus/}.

\subsection{Definitions and Basics}
\indent Two matrices $X$ and $Y$ of order $n$ are said to be \textit{amicable} if $XY^\top = YX^\top$.

We define the following classes of propus like matrices. We note that there are slight variations in the matrices which allow variant arrays and non-circulant matrices to be used to give symmetric Hadamard matrices, All propus like matrices $A$, $B=C$, $D$ are $\pm 1$ matrices of order $n$  satisfy the \textit{additive property} 
\begin{equation}\label{additive}
                AA^{\top} + 2BB^{\top}  + DD^{\top} = 4nI_n ,
\end{equation} 
$I$ the identity matrix, $J$ the matrix of all ones. 

We make the following definitions: 
\begin{itemize}
\item \textit{propus matrices}: four circulant symmetric $\pm 1$ matrices, $A$, $B$, $B$, $D$ of order $n$, satisfying the additive property (use $P$);
\item \textit{propus-type matrices}: four pairwise amicable $\pm 1$ matrices, $A$, $B$, $B$, $D$ of order $n$, $A^{\top } = A$, satisfying the additive property (use $P$);
\item \textit{generalized-propus matrices}: four pairwise commutative $\pm 1$ matrices, $A$, $B$, $B$, $D$ of order $n$, $A^{\top } = A$, which satisfy the additive property (use $GP$).
\end{itemize}

 We use two types of arrays into which to plug the propus like matrices:
 the Propus array, $P$, or the generalized-propus array, $GP$. 
 These can also be used  with generalized matrices (\cite{JSW1974}).
  \begin{center}
 \begin{tabular}{ccc}  
$ P =  \begin{bmatrix*}[r]
    A&  B  &  B  & D \\
    B&  D  & -A  &-B \\
    B& -A  & -D  & B \\
    D& -B  &  B  &-A 
    \end{bmatrix*} $ & and & 
 $ GP =  \begin{bmatrix*}[r]
     A&  BR  &  BR  & DR \\
     BR&  D^\top R  & -A  &-B^\top R \\
     BR& -A  & -D^\top R  & B^\top R \\
     DR& -B^\top R  &  B^\top R  &-A. 
     \end{bmatrix*}.  $ 
     \end{tabular} \end{center}
     
Symmetric Hadamard matrices made using propus like matrices will be called \textit{symmetric propus-Hadamard matrices}.    

\section{Symmetric Propus-Hadamard Matrices}
We first give the explicit statements of two well known theorem, Paley's Theorem \cite{RP1933}, for the Legendre core $Q$, and Turyn's Theorem \cite{RT1972}, in the form in which we will use them.

\begin{theorem} \label{th:paley} {\rm \textbf{[Paley's Legendre Core \cite{RP1933}]}}
Let  $p$ be a prime power, either $\equiv 1 \pmod{4}$ or $\equiv 3 \pmod{4}$ then there exists a matrix, $Q$, of order $p$ with zero diagonal and other elements ${\pm 1}$
satisfying $QQ^{\top } = (q+1)I - J$, $Q$ is symmetric or skew-symmetric according as $p \equiv 1 \pmod{4}$ or $p \equiv 3 \pmod{4}$. 
\end{theorem}

\begin{theorem} \label{th:turyn} {\rm \textbf{[Turyn's Theorem \cite{RT1972}]}}
Let $q \equiv 1 \pmod {4}$ be a prime power then there are two symmetric  matrices, $P$ and $S$ of order $\frac12(q+1)$, satisfying $PP^{\top }+ SS^{\top } = qI$: $P$ has zero diagonal and other elements ${\pm 1}$ and $S$ elements ${\pm 1}$.
\end{theorem}

\subsection{Propus-Hadamard Matrices from Williamson Matrices}
\begin{lemma} \label{lem:basic} Let $q \equiv 1 \pmod{4}$, be a prime power, then propus matrices exist for orders 
$n = \frac{1}{2}(q+1)$ which give symmetric propus-Hadamard matrices of order $2(q+1)$. 
 \end{lemma}

\begin{proof}
We note that for $q \equiv 1 \pmod{4}$, a prime power, Turyn (Theorem \ref{th:turyn} \cite{RT1972}) gave Williamson matrices, $X+I$, $X-I$, $Y$, $Y$, which are circulant and symmetric for orders $n = \frac12(q+1)$. Then choosing
$$ A= X+I, ~ B = C = Y,~ D=X-I$$
gives the required propus-Hadamard matrices.
\end{proof}

We now have propus-Hadamard matrices for orders $4n$ where $n$ is in 
\begin{multline*}
\{1,3,[5],7,9,[13],15,19,21,[25],27,31,37,[41],45,49,51,55,57,59,\\
[61],[63],
67,69,75,79,81,[85],87,89,91,97,99,105,111,115,117,119,121,\\
127,129,135,139,141,[145],147,157,159,169,175,177,[181],187,195,199. \}
\end{multline*}

The cases written in square brackets [5],[13],[25],[41],[61],[63],[85],[113],\newline [145],[181] arise when $q$ is a  prime power, however the Delsarte-Goethals-Seidel-Turyn construction means the required circulant matrices also exist for these prime powers (see Figure \ref{fig:P20-P52}). 

\begin{figure}[H]
  \centering
\subfloat[][P12 ($q=5;n=3$)]{\includegraphics[width=0.27\textwidth]{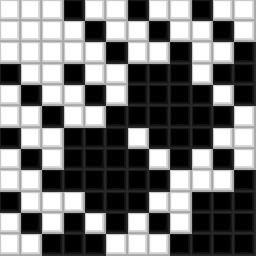}} \quad
%\subfloat[][P28 ($q=13;n=7$)]{\includegraphics[width=0.27\textwidth]{P28CL2}}\\
\subfloat[][P28 ($q=13;n=7$)]{\includegraphics[width=0.27\textwidth]{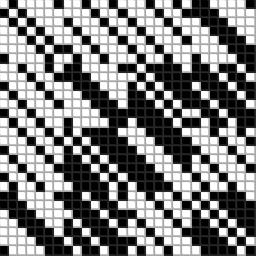}}\\
  \caption{Propus-Hadamard matrices for orders $4q$}
  \label{fig:P12-P28-P36-P60}
\end{figure}
  
\subsubsection{Propus matrices of small order and from $q$ prime power}

\begin{figure} [H]
  \centering
  \subfloat[][P20 ($q=3^2;n=5$)]{\includegraphics[width=0.27\textwidth]{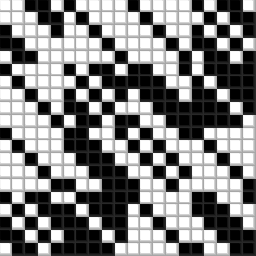}} \quad
  %\subfloat[][P52 ($q=5^2;n=13$)]{\includegraphics[width=0.27\textwidth]{P52CL6}}
  \subfloat[][P52 ($q=5^2;n=13$)]{\includegraphics[width=0.27\textwidth]{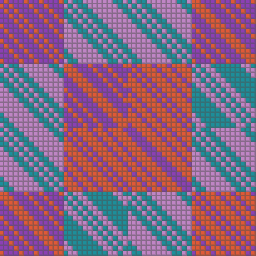}}
  \caption{Propus-Hadamard matrices for orders $4q$ $q$ a prime power.}
  \label{fig:P20-P52}
\end{figure}  

This family is considered to contain the two trivial propus-Hadamard matrices of orders 12 and 20 based on symmetric Paley cores 
$A=J$, $B = C = J-2I = QR$, $D=J=2I = QR$  for $n=3$, and $A = Q + I$, $B=C=J-2I$, $D = Q - I$ (constructed using Legendre symbols)  for n=5.  This special set can be continued with back-circulant matrices $C=B$ which allows the symmetry property of $A$ to be conserved.

\subsection{Propus-Hadamard matrices from $D$-optimal designs}
\begin{lemma}
Let $n \equiv 3 \pmod{4}$, be a prime, such that $D$-optimal designs,
constructed using two circulant matrices, one of which is symmetric, exist for
order $2n$. Then propus-Hadamard matrices exist for order $4n$.
\end{lemma}

Djokovi\'{c} and Kotsireas in \cite{AGJS1979,DDIK2012} give $D$-optimal designs, constructed using two circulant matrices, for 
$n \in \{3,5,7,9,13,15,19,21,23,25,27,31,33,\newline
37,41,43,45,49,51,55,57,59,61,63,69,73,75,77,79,85,87,91,93,97,103,\newline 113,121, 131,133,145,157,181,183\},$ $n < 200$. 
We are interested in those cases where the $D$-optimal design is constructed from two circulant matrices one of which must be symmetric.

Suppose $D$-optimal designs for orders $n \equiv 3 \pmod{4}$, a prime, 
are constructed using two circulant matrices, $X$ and $Y$. Suppose $X$ is symmetric. Let $Q+I$ be the Paley matrix of order $n$. Then choosing
$$ A= X,\quad B = C = Q+I, \quad D=Y,$$
to put in the array $GP$ gives the required propus-Hadamard matrices.

Hence we have propus-Hadamard matrices, constructed using $D$-optimal designs, 
for orders $4n$ where $n$ is in
\[\{3,7,19,31\}.\] 
The results for $n$ = 19 and 31  were given to us by Dragomir Djokovi\`{c}.
 
\begin{figure}[H]
  \centering
  \subfloat[][D6 ($n=3$)]{\includegraphics[width=0.27\textwidth]{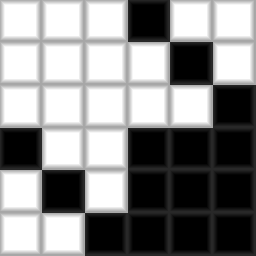}} \quad
  \subfloat[][GP12 ($n=3$)]{\includegraphics[width=0.27\textwidth]{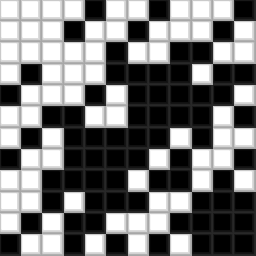}}\\
  \subfloat[][D14 ($n=7$)]{\includegraphics[width=0.27\textwidth]{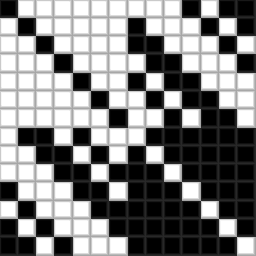}} \quad
  \subfloat[][GP28 ($n=7$)]{\includegraphics[width=0.27\textwidth]{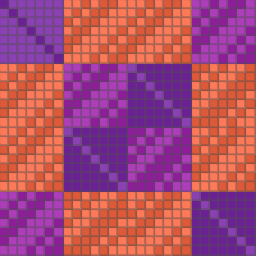}}\\
  \subfloat[][D38 ($n=19$)]{\includegraphics[width=0.27\textwidth]{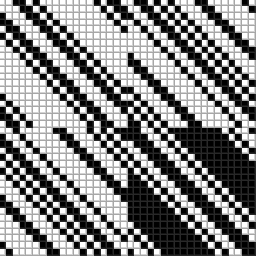}} \quad
  \subfloat[][GP76 ($n=19$)]{\includegraphics[width=0.27\textwidth]{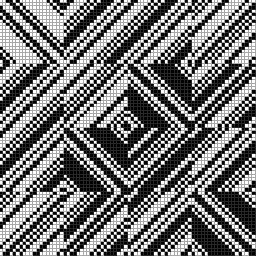}}\\
  \caption{D-optimal designs for orders $2n$ propus-Hadamard matrices for orders $4n$}
  \label{fig:D6-GP12-GP28-GP76}
\end{figure}

\subsection{A Variation of a Theorem of Miyamoto}

In Seberry and Yamada \cite{SY92} one of Miyamoto's  
results \cite{miyamoto1991} was reformulated so that symmetric Williamson-type 
matrices can be obtained.  The results given here are due to 
Miyamoto, Seberry and Yamada.

\begin{lemma}[Propus Variation] \label{le:1}
Let $U_i$, $V_j$, $i, j=1, 2, 3, 4$ be $(0, +1, -1)$ matrices of order $n$ 
which satisfy
\begin{enumerate}
\item[(i)] $U_i$, $U_j$, $i \ne j$ are pairwise amicable,
\item[(ii)] $V_i$, $V_j$, $i \ne j$ are pairwise amicable,
\item[(iii)] $U_i \pm V_i$, $(+1, -1)$ matrices, $i = 1, 2, 3, 4$,
\item[(iv)] the row sum of $U_1$ is 1, and the row sum of $U_j$, $i = 2,3,4$ 
is zero,
\item[(v)] $\sum_{i=1}^4 U_iU_i^T = (2n+1)I - 2J$, 
  $\sum_{i=1}^4 V_iV_i^T = (2n+1)I$.
\end{enumerate}

Let $S_1$, $S_2$, $S_3$, $S_4$ be four $(+1, -1)$-matrices of order $2n$ defined by
\begin{displaymath}
S_j = U_j \times \left[ \begin{array}{cc} 1 & 1 \\
      1 & 1 \end{array} \right] 
 + V_j \times \left[ \begin{array}{rr} 1 & -1 \\
      -1 & 1 \end{array} \right],
\end{displaymath}
where $S_2 = S_3$. 

Then there are 4  propus-Williamson type matrices of order $2n+1$. 
If $U_i$ and $V_i$ are symmetric, $i = 1, 2, 3, 4$ then the Williamson-type 
matrices are symmetric.
Hence there is a symmetric propus-type Hadamard matrix of order $4(2n+1)$.
\end{lemma}

\begin{proof} 
With $S_1$, $S_2$, $S_3$, $S_4$, as in the theorem enunciation
the row sum of $S_1 = 2$ and of $S_i = 0$, $i = 2, 3, 4$. Now define
\begin{displaymath}
X_1 = \left[ \begin{array}{cc} 1 & -e_{2n} \\
    -e_{2n}^T & S_1 \end{array} \right] 
\mbox{\quad and\quad} 
X_i = \left[ \begin{array}{cc} 1 & e_{2n} \\
    e_{2n}^T & S_i \end{array} \right],
\quad i = 2, 3, 4. 
\end{displaymath}
First note that since $U_i$, $U_j$, $i \ne j$ and $V_i$, $V_j$, 
$i \ne j$ are pairwise amicable,
\begin{align*}
S_i S_j^T &= \left(\! U_i \times 
  \left[\! \begin{array}{cc} 1 \!&\! 1 \\
             1 \!&\! 1 \end{array} \!\right] 
+ V_i \times \left[\! \begin{array}{rr} 1 \!&\!\! -1 \\ 
                  -1 \!&\!\! 1 \end{array} \!\right] \right) 
  \left(\! U_j^T \times\! \left[\! \begin{array}{cc} 1 \!&\! 1 \\
                      1 \!&\! 1 \end{array} \!\right] 
+ V_j^T \times\! \left[\! \begin{array}{rr} 1 \!&\!\! -1 \\
                 -1 \!&\!\! 1 \end{array} \!\right] \right)\\
&= U_iU_j^T \times \left[ \begin{array}{cc} 2 & 2 \\
                   2 & 2 \end{array}\right]
+ V_iV_j^T \times \left[\begin{array}{rr} 2 & -2 \\
                  -2 & 2 \end{array}\!\right]\\
&= S_jS_i^T.
\end{align*}
(Note this relationship is valid if and only if conditions (i) and
(ii) of the theorem are valid.)

\begin{align*}
\sum_{i=1}^4 S_i S_i^T
&= \sum_{i=1}^4 U_iU_i^T \times 
  \left[ \begin{array}{cc} 2 & 2 \\
               2 & 2 \end{array} \right] 
+ \sum_{i=1}^4 V_iV_i^T \times 
  \left[\begin{array}{rr} 2 & -2 \\
               -2 & 2 \end{array} \right] \\
&= 2 \left[ \begin{array}{cc} 2(2n+1)I -2J & -2J \\ 
                   -2J & 2(2n+1)I -2J\end{array} \right] \\
&= 4(2n + 1) I_{2n} - 4J_{2n} 
\end{align*}

Next we observe
\begin{displaymath}
X_1 X_i^T = \left[ \begin{array}{cc} 1 -2n & e_{2n} \\ 
                 e_{2n}^T & -J + S_1 S_i^T \end{array} \right] 
= X_i X_1^T \qquad i = 2, 3, 4,
\end{displaymath}
and
\begin{displaymath}
X_i X_j^T = \left[ \begin{array}{cc} 1+2n & e_{2n} \\ 
                  e_{2n}^T & J + S_i S_j^T \end{array} \right] 
= X_j X_i^T \qquad i \ne j,\ \ i,j = 2, 3, 4.
\end{displaymath}
Further

\begin{align*}
\sum_{i=1}^4 X_i X_i^T 
&= \left[\begin{array}{cc} 1+2n & -3e_{2n} \\
              -3e_{2n}^T & J + S_1 S_1^T \end{array} \right] 
+ \sum_{i=2}^4 \left[\begin{array}{cc} 1+2n & e_{2n} \\ 
                e_{2n}^T & J + S_i S_i^T \end{array}\right]\\
&= \left[\begin{array}{cc} 4(2n+1) & 0 \\
               0 & 4J + 4(2n+1)I - 4J \end{array}\right] .
\end{align*}
 
Thus we have shown that $X_1$, $X_2$, $X_3$, $X_4$ are pairwise amicable, symmetric Williamson 
type matrices of order $2n+1$, where $X_2 = X_3$. These can be used as in (ii) of Theorem using the additive property to obtain the required symmetric propus Hadamard matrix of order $(4(2n+1)$.
\end{proof}

Many powerful corollaries arose and new results were obtained by making 
suitable choices in the theorem.  We choose $X_1$, $X_2$, $X_3$, $X_4$ to ensure that the propus construction can be used to form symmetric Hadamard matrices of order $4(2n+1)$.

From Paley's theorem (Corollary \ref{th:paley}) for $p \equiv 3 \pmod{4}$ we use the backcirculant or type 1, symmetric matrices $QR$ and $R$ instead of $Q$ and $I$; whereas for $p \equiv 1 \pmod{4}$ we use the symmetric Paley core $Q$. If $p$ is a prime power $\equiv 3 \pmod{4}$ we set $U_1=I$, $U_2 = U_3 = QR$, $U_4 = 0$ of order $p$, and if $p$ is a prime power $\equiv 1 \pmod{4}$, we set $U_1=I$, $U_2 = U_3 = Q$, $U_4 = 0$ of order $p$. Hence $\sum_{k=1}^4 V_kV_k^{\top } = (q+2)I$.

From Turyn's result (Corollary \ref{th:turyn}) we set, for $p \equiv 1 \pmod{4}$ $U_1=P$, $U_2= U_3 = I$ and $U_4=S$, and for
$p \equiv 3 \pmod{4}$, $V_1=P$, $V_2= V_3 = R$ and $V_4=S$, so $\sum_{k=1}^4 U_kU_k^{\top } = (q+2)I$.

Hence we have:

\begin{corollary} \label{cor:impt}
Let $q \equiv 1 \pmod{4}$ be a prime power and $\frac12(q+1)$ be a prime power or the order of the core of a 
symmetric conference matrix (this happens for $q=89$).   Then there exist
symmetric Williamson type matrices of order $2q+1$ and a symmetric propus-type Hadamard matrix
of order $4(2q+1)$.  \end{corollary} 

This gives the previously unresolved cases for 11 and 83.

\subsubsection{Three Equal}
The family given above includes two starting Hadamard matrices of orders 12 and 28  based on the skew Paley core $B =C = D = Q + I$ (constructed using Legendre symbols). This special set is finite because $12 = 3^2+1^2 +1^2+1^2$ and $28 = 5^2+1^2+1^2+1^2$ and these are the only orders for which a symmetric circulant $A$ can exist with $B=C=D$.

\begin{figure}[H] 
  \centering
\subfloat[][P12skew ($n=3$)]{\includegraphics[width=0.25\textwidth]{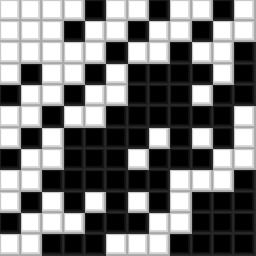}} \quad
\subfloat[][P28skew ($n=7$)]{\includegraphics[width=0.25\textwidth]{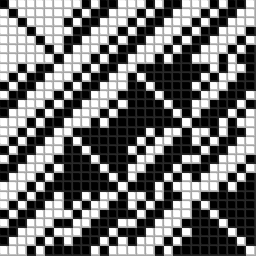}}
  \caption{Propus-Hadamard matrices using $D$-optimal designs}
  \label{fig:P12skew-and-P28skew}
\end{figure}  

\section{Propus-Hadamard \hfill matrices\hfill from \hfill conference
\newline matrices: even order matrices}

A powerful method to construct propus-Hadamard matrices for $n$ even is using conference matrices.

\begin{lemma} Suppose $M$ is a conference matrix of order $n \equiv 2 \pmod{4}$. Then $MM^\top = M^\top M =(n-1)I,$ where $I$ is the identity matrix and $M^\top =M$. Then using $A=M+I,~B=C=M-I,~D=M+I$ gives a propus-Hadamard matrix of order $4n$.
\end{lemma}
 We use the conference matrix orders from \cite{NBJS2014} and so have propus-Hadamard matrices of orders $4n$ where $n \in $
 $$\{6,10,14,18,26,30,38,42,46,50,54,62,74,82,90,98\}.$$

The conference matrices in Figure \ref{fig:C10-P20-C26-P52} are made two circulant matrices $A$   and $B$ of order $n$ where both $A$ and $B$ are symmetric.

Then using the matrices $A+I$, $B = C$ and $D= A-I$ in $P$ gives the required construction. 

\begin{figure}[H] 
  \centering
  \subfloat[][CP20 ($n=3$)]{\includegraphics[width=0.27\textwidth]{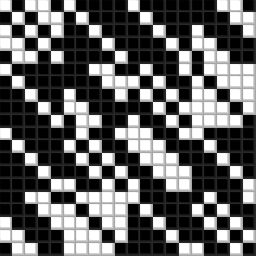}}\\
  \subfloat[][C26 ($n=13$)]{\includegraphics[width=0.27\textwidth]{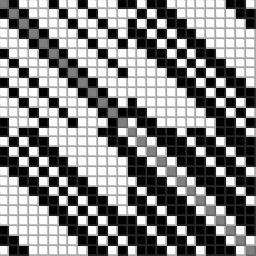}} \quad
  \subfloat[][CP52 ($n=13$)]{\includegraphics[width=0.27\textwidth]{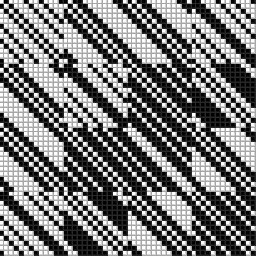}}\\
  \caption{Conference matrices for orders $2n$ using two circulants: propus-Hadamard matrices for orders $4n$}
  \label{fig:C10-P20-C26-P52}
\end{figure}
  
The conference matrices in Figure \ref{fig:C10-GP20-C26-GP52} are made from two circulant matrices $A$ and $B$ of order $n$ where both $A$ and $B$ are symmetric. However here we use $A+I$, 
$BR = CR$ and $D= A-I$ in $P$ to obtain the required construction. 

\begin{figure}[H] 
  \centering
  %\subfloat[][C10G ($n=5$)]{\includegraphics[width=0.27\textwidth]{C10G}} \quad
  %\subfloat[][CP20G ($n=5$)]{\includegraphics[width=0.27\textwidth]{CP20G}}\\
  \subfloat[][C26G ($n=13$)]{\includegraphics[width=0.27\textwidth]{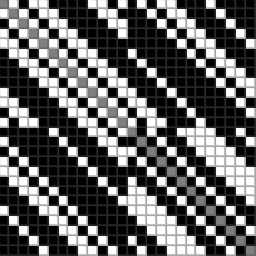}} \quad
  \subfloat[][CP52G ($n=13$)]{\includegraphics[width=0.27\textwidth]{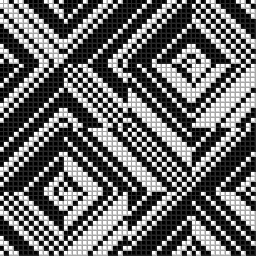}}\\
  \caption{Conference matrices for orders $2n$ using two circulant and back-circulants: propus-Hadamard matrices for orders $4n$}
  \label{fig:C10-GP20-C26-GP52}
\end{figure}  
There is another variant of this family which uses the symmetric Paley cores $A = Q + I$, $D=Q - I$ (constructed using Legendre symbols) and one circulant matrix of maximal determinant $B=C=Y$.

\subsection{Propus-Hadamard matrices for $n$ even}

Figure \ref{fig:P16-P32} gives visualizations (images/pictures) of propus-Hadamard matrices orders  16, 32. These have even $n$. 

\begin{figure}[H] \label{fig7}
  \centering
   \quad
  \subfloat[][P16 ($n=4$)]{\includegraphics[width=0.27\textwidth]{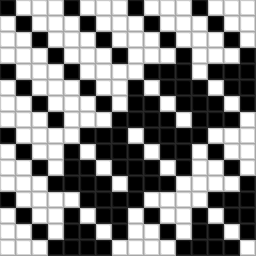}} \quad
{\tiny }  
  %\subfloat[][P32 ($n=8$)]{\includegraphics[width=0.27\textwidth]{P32CL4}}
  \subfloat[][P32 ($n=8$)]{\includegraphics[width=0.27\textwidth]{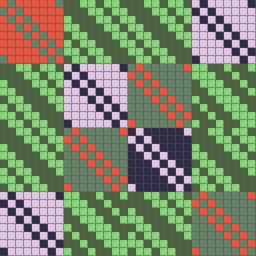}}
  \caption{Matrices P16 and P32}
  \label{fig:P16-P32}
\end{figure}

 \section{Conclusion and Future Work}
 Using the results of Lemma \ref{lem:basic} and Corollary \ref{cor:impt} 
 and the symmetric propus-Hadamard matrices of Di Matteo,  Djokovi\'{c}, and  Kotsireas given in \cite{MDK2015}, we see that the unresolved cases for symmetric propus-Hadamard matrices for orders $4n$, $n < 200$ odd, are where $n \in $
 \begin{multline*}
 \{17,23,29,33,35,39,47,53,65,71,73,77,93,95,97,99,\\
 101,103,107,109,113,125,131,133,137,143,149,151,153,155,\\
  161,163,165,167,171,173,179,183,185,189,191,193,197.\}
 \end{multline*} 
 
There are many constructions and variations of the propus theme to be explored in future research.
Visualizing the propus construction gives aesthetically pleasing examples of propus-Hadamard matrices.
The visualization also makes the construction method clearer. There is the possibility that these visualizations may be used for quilting.

\end{document}